\documentclass[12pt]{article}

\usepackage{amsmath}\date{}

\usepackage{latexsym}
\usepackage{eepic}
\usepackage{makeidx}
\usepackage[dvips]{graphicx}
\usepackage[english]{babel}
\usepackage{amssymb}
\usepackage{amsmath}
\usepackage[margin=2.5cm]{geometry}

\usepackage{amsthm}

\usepackage{fancyhdr}

\newtheorem{theorem}{Theorem}[section]
\newtheorem{lemma}[theorem]{Lemma}
\newtheorem{corollary}[theorem]{Corollary}
\newtheorem{proposition}[theorem]{Proposition}
\newtheorem{conjecture}[theorem]{Conjecture}

\theoremstyle{definition}

\begin{document}
\date{\today}
\title{Extending partial edge colorings of cartesian products of graphs}

\author{
Carl Johan Casselgren\footnote{Department of Mathematics, 
Link\"oping University, 
SE-581 83 Link\"oping, Sweden.
{\it E-mail address:} carl.johan.casselgren@liu.se.
Research supported by a grant from the Swedish Research council VR
(2017-05077).
}
\and
Fikre B. Petros\footnote{Department of Mathematics, 
Addis Ababa University,
1176 Addis Ababa, Ethiopia.
{\it E-mail address:} fikre.bogale@aau.edu.et}
\and
Samuel A. Fufa\footnote{Department of Mathematics, 
Addis Ababa University,
1176 Addis Ababa, Ethiopia.
{\it E-mail address:} samuel.asefa@aau.edu.et}
}
\maketitle

\bigskip
\noindent
{\bf Abstract.}
We consider the problem of extending partial edge colorings of cartesian
products of graphs. More specifically, we suggest the following
Evans-type conjecture: If $G$ is a 
graph where every precoloring of at most 
$k$ precolored edges can be extended to a proper $\chi'(G)$-edge coloring, 
then every precoloring of at most $k+1$ edges of $G \square K_2$ is extendable	
to a proper $(\chi'(G) +1)$-edge coloring of $G \square K_2$.
In this paper we verify that this conjecture holds for trees, complete
and complete bipartite graphs, as well as for graphs with 
small maximum degree.
We also prove versions of the conjecture for general regular graphs
where the precolored edges are required to be independent.

\bigskip


\section{Introduction}
	An {\em edge precoloring} (or {\em partial edge coloring})
	of a graph $G$ is a proper edge coloring of some
	subset $E' \subseteq E(G)$; {\em a $t$-edge precoloring}
	is such a coloring with $t$ colors.
	An edge $t$-precoloring $\varphi$ is
	{\em extendable} if there is a proper $t$-edge coloring $f$
	such that $f(e) = \varphi(e)$ for any edge $e$ that is colored
	under $\varphi$; $f$ is called an {\em extension}
	of $\varphi$.
	In general, the problem of extending a given edge precoloring
	is an $\mathcal{NP}$-complete problem,
	already for $3$-regular bipartite graphs \cite{Fiala}.
	
	Edge precoloring extension problems seem to have been 
	first considered in connection with the problem of completing partial
	Latin squares and the well-known Evans' conjecture 
	that every $n \times n$ partial Latin square with at most $n-1$ non-empty
	cells is completable to a partial Latin square \cite{Evans}. By a well-known 
	correspondence, the problem of completing
	a partial Latin square is equivalent to asking if a partial 
	edge coloring with 
	$\Delta(G)$ colors of a balanced complete bipartite graph $G$ is extendable
	to a $\Delta(G)$-edge coloring, where $\Delta(G)$ as usual denotes the maximum degree.
	 Evans' conjecture was proved for large $n$ by H\"aggkvist \cite{Haggkvist78},
	and in full generality by Andersen and Hilton \cite{AndersenHilton}, and, 
	independently, by
	Smetaniuk \cite{Smetaniuk}.
	
	Another early reference on
	edge precoloring extension is \cite{MarcotteSeymour}, where the authors 
	study the problem from the viewpoint of polyhedral combinatorics.
	More recently, the problem of extending a precoloring of
	a matching has been considered
	in \cite{EGHKPS}.
	In particular, it is conjectured that for every graph $G$,
	if $\varphi$ is an edge precoloring of a matching $M$ in $G$
	using $\Delta(G)+1$ colors,
	and any two edges in $M$ 
	are at distance at least $2$ from each other,  then $\varphi$ 
	can be extended to a proper $(\Delta(G)+1)$-edge coloring of $G$;
	here,
	by the {\em distance} between two
	edges $e$ and $e'$ we mean the number of edges in a 
	shortest path between an endpoint of $e$ and an endpoint of $e'$; 
	a {\em distance-$t$ matching} is a matching where any
	two edges are at distance at least $t$ from each other.
	In \cite{EGHKPS}, it is proved that this conjecture holds
	for e.g. bipartite multigraphs and subcubic multigraphs, and
	in \cite{GiraoKang} it is proved that a version of the conjecture with the
	distance increased to 9 holds for general graphs.

	Quite recently, with motivation
	from results on completing partial Latin squares,
	questions on extending partial edge colorings
	of $d$-dimensional hypercubes $Q_d$ were studied 
	in \cite{CasselgrenMarkstromPham}. 
	Among other things, a characterization of
	partial edge colorings with at most $d$ precolored edges are extendable
	to $d$-edge colorings of $Q_d$ is obtained, thereby establishing an analogue
	for hypercubes of the characterization by
	Andersen and Hilton \cite{AndersenHilton} of 
	$n \times n$ partial Latin 
	squares with at most $n$ non-empty cells that are completable to Latin squares.
	In particular, every partial $d$-edge coloring with at most $d-1$ colored edges
	is extendable to a $d$-edge coloring of $Q_d$.
	This line of investigation was continued in 
	\cite{CasselgrenPetros, CasselgrenPetros2} where similar questions are
	investigated for trees.
	
	Denote by $G \square H$ the cartesian product of the graphs $G$ and $H$.
	Motivated by the result on hypercubes \cite{CasselgrenMarkstromPham}, which
	are iterated cartesian products of $K_2$ with itself,
	in this paper we continue the investigation of edge precoloring extension
	of graphs with a particular focus on Evans-type questions for
	cartesian products of graphs.
	We are particularly interested in the following conjecture, which would be 
	a far-reaching generalization of a main result of \cite{CasselgrenMarkstromPham}.

	\begin{conjecture}
\label{conj:general}
	If $G$ is a graph where every precoloring of at most $k$
	edges can be extended to a proper $\chi'(G)$-edge coloring,
	then every precoloring of at most $k+1$ edges of $G \square K_2$ is extendable
	to a proper $(\chi'(G) +1)$-edge coloring of $G \square K_2$.
\end{conjecture}
	
As we shall see, Conjecture \ref{conj:general} becomes false
if we replace $(\chi'(G) +1)$ with $\chi'(G \square K_2)$ by the
example of odd cycles.
\bigskip

It it straightforward that if every  edge precoloring with $\Delta(G)$
colored edges of a connected Class 1 graph $G$ is extendable to a $\Delta(G)$-edge coloring,
then $G$ is isomorphic to a star $K_{1,n}$. If $G$, on the other hand,
is a connected Class 2 graph where any  edge precoloring with 
$\Delta(G) +1$ colored
edges is extendable, then $G$ is an odd cycle.
Therefore, 
we shall generally only consider
precolorings with at most $\chi'(G)-1$ colored edges.

In fact, it is easy to show that odd cycles and stars are 
the only connected
graphs with the property that any partial $\chi'(G)$-edge coloring is
extendable.

\begin{proposition}
\label{prop:alwaysextend}
	Every partial $\chi'(G)$-edge coloring of $G$ is extendable if and only
	$G$ is isomorphic to a star $K_{1,n}$ or and odd cycle.
\end{proposition}

In this paper, 
we verify that Conjecture \ref{conj:general}
holds for trees, complete  and complete bipartite graphs.
Moreover, we
prove a version of Conjecture \ref{conj:general}
for regular triangle-free graphs where the precolored edges are required
to be independent; a version for graphs with triangles is proved as well.
Finally, we prove it for graphs with small maximum degree, namely, 
graphs with maximum degree two and Class 1 graphs with maximum degree
$3$.

\section{Cartesian products with trees, complete graphs and complete bipartite graphs}

In this section, we prove that Conjecture \ref{conj:general}
holds for trees, complete bipartite graphs and complete graphs.
In the following we say that an edge $e$ is {\em $\varphi$-colored} if
$\varphi$ is a (partial) edge coloring and $e$ is colored under $\varphi$.
A color $c$ {\em appears} at a vertex $v$ if some edge incident with $v$ is colored $c$;
otherwise $c$ is {\em missing} at $v$.

Let us first consider trees.
In \cite{CasselgrenPetros}, it was proved that if at most $\Delta(T)$ edges are properly precolored in a tree $T$, then the partial coloring $\varphi$ can be extended to a proper $\Delta(T)$-edge-coloring of $T$ unless $\varphi$ satisfies any of the following conditions:\\
 
\noindent $($C1$)$ there is an uncolored edge $uv$ in $T$ such that $u$ is incident with edges of $k < \Delta(T)$ distinct colors and $v$ is incident to $\Delta(T) - k$ edges colored with $\Delta(T) - k$ other distinct colors (so $uv$ is adjacent to edges of $\Delta(T)$ distinct colors);\\

\noindent $($C2$)$ there is a vertex $u$ of degree $\Delta(T)$ that is incident with edges of $\Delta(T) - k$ distinct colors $c_1,\ldots,c_{\Delta(T)-k}$, and $k$ vertices $v_1, \ldots, v_k$, where $1 \leq k < \Delta(T)$,
 such that for $i=1, \ldots, k$, $uv_i$ is uncolored but $v_i$ is incident with an edge colored by a fixed color $c \notin \lbrace c_1, \ldots, c_{\Delta(T)-k}\rbrace$;\\

\noindent$($C3$)$ there is a vertex $u$ of degree $\Delta(T)$ such that every edge incident with $u$ is uncolored but there is a fixed color $c$ satisfying that every edge incident with $u$ is adjacent to another edge colored $c$;\\

\noindent $($C4$)$ $\Delta(T)=2$ and there are two precolored edges using the same color if they are at even distance, and using different colors if they are at odd distance. \\

 The set of all colorings of $T$ with $\Delta(T) \geq 2$ satisfying the corresponding conditions above are denoted by $\mathcal{C}_i$ for $i=1, 2, 3, 4$ 
and we define $\mathcal{C} = \cup \mathcal{C}_i$. 
We state the result of \cite{CasselgrenPetros} as the following theorem.

\begin{theorem}
\label{thm1}
\cite{CasselgrenPetros}
\textit{Let $T$ be a forest with maximum degree $\Delta(T)$. If $\varphi$ is  a
 $\Delta(T)$-edge precoloring of $T$ with at most $\Delta(T)$ precolored edges and $\varphi \notin \mathcal{C} $, then 
$\varphi$ is extendable to a proper $\Delta(T)$-edge coloring of $T$.}
\end{theorem}

The following is an immediate consequence of Theorem \ref{thm1}.

\begin{corollary}
\cite{CasselgrenPetros}
\textit{If $T$ is a forest with maximum degree 
$\Delta(T)$, then any partial edge coloring   with at most $\Delta(T) - 1$ precolored edges is  extendable to a proper $\Delta(T)$-edge
coloring of $T$.}
\label{cor1}
\end{corollary}

Using these results we shall establish Conjecture \ref{conj:general} for the case of
trees.

\begin{theorem}
\label{th:tree}
\textit{Let $T$ be a tree with maximum degree $\Delta(T)$. If $\varphi$ is an  edge precoloring of $\Delta(T)$ edges in $T \square K_2$, then $\varphi$ can be extended to a proper $(\Delta(T) + 1)$-edge coloring of $T \square K_2$.}
\end{theorem}

\begin{proof}
Let $M$ be a perfect matching in $T \square K_2$ such that $T \square K_2 - M$ is isomorphic to two copies $T_1$ and $T_2$ of $T$. 
Without loss of generality we assume that the precoloring of $T \square K_2$ uses colors $1, \ldots, \Delta(T)$. We shall consider the following two different cases: the case when $M$ contains precolored edges, and the case when it does not. Here and henceforth, two edges are \textit{corresponding} if their endpoints are joined by two edges of $M$. Similarly, two vertices are \textit{corresponding} if they are joined by an edge of $M$. \\

\noindent \textbf{Case 1.} \textit{No precolored edges are in $M$:}\\

If no precolored edges are in $T_2$, then since $\Delta(T)$ edges are precolored, by Corollary \ref{cor1} there is a proper $(\Delta(T) + 1)$-edge coloring of $T_1$ using colors $1, \ldots, \Delta(T) +1$ that agrees with $\varphi$. Hence, we obtain an extension of $\varphi$ by coloring every edge of $T_2$ by the color of its corresponding edge in $T_1$, and then coloring every edge of $M$ by some color not appearing on any edge incident to one of its endpoints. 

If, on the other hand, $T_1$ contains at most $\Delta(T) -1$ precolored edges, then by Corollary \ref{cor1} there is a proper edge coloring $\varphi_1$ of $T_1$ using colors $1, \ldots, \Delta(T)$ which is an extension of the restriction of $\varphi$ to $T_1$. Similarly, there is a proper edge coloring $\varphi_2$ of $T_2$ using colors $1, \ldots, \Delta(T)$ which is an extension of the restriction of $\varphi$ to $T_2$. Hence, by coloring the edges of $M$ with color $\Delta(T) +1$,
 we obtain a proper $(\Delta(T) +1)$-edge coloring of $T \square K_2$. \\

\noindent \textbf{Case 2.} \textit{$M$ contains at least one precolored edge:}\\

By slight abuse of notation, we denote by $T_i +M$ the graph obtained from $T_i$
by attaching every edge of $M$ as a pendant edge of $T_i$.

If no precolored edges are in $T_2$, then since $T_1 + M$ is a forest 
with $\Delta(T)$ precolored edges, by Corollary~\ref{cor1} there is a 
proper $(\Delta(T) +1)$-edge coloring of $T_1 + M$ that agrees with 
$\varphi$. Hence, by coloring every edge of $T_2$ by the color of its corresponding edge in $T_1$, we obtain a proper $(\Delta(T) +1)$-edge coloring of $T \square K_2$. 

Suppose now that both $T_1$ and $T_2$ contains at least one precolored edge. 
Set $G_i = T_i +  M$ and consider the restriction of $\varphi$ to $G_1$ and $G_2$,
respectively. Note that both $G_1$ and $G_2$ are forests that each contains at 
most $\Delta(T) -1$ precolored edges. 
We shall define extensions of these precolorings of $G_1$ and $G_2$, respectively,
which agrees on $M$. This yields an extension of $\varphi$.

Let $V^1_M$ be the set containing all 
vertices of degree $\Delta(T) +1$ in $G_1$ that are incident with some precolored 
edge from $M$ and $V^2_M$ be the set containing all vertices of degree $\Delta +1$ 
in $G_2$ that are incident with some precolored edge from $M$. Note that 
$v_1 \in V^1_M$ if and only if the corresponding vertex $v_2 \in V^2_M$.
We shall prove that there are matchings $M_1 \subseteq E(T_1)$ and $M_2 \subseteq E(T_2)$  such that $M_1$ covers $V^1_M$ and $e_2 \in M_2$ if and only if the corresponding edge $e_1 \in M_1$. Moreover we require that if $u_1v_1 \in M_1$ and the corresponding edge $u_2v_2 \in M_2$, then 
\begin{itemize}

\item[(i)] exactly one of $u_1$ and $v_1$ is in $V^1_M$, say $u_1$;

\item[(ii)] the components $H_1$ of $G_1 - u_1$ and $H_2$ of $G_2 - u_2$ containing $v_1$ and $v_2$, respectively, contain no precolored edges;

\item[(iii)] $M_1$ and $M_2$ do not contain any precolored edges. 
 
\end{itemize}
Since all vertices of $V^i_M$ have degree $\Delta(T)$ in $T_i$,
$T_1 \cup T_2$ contains altogether at most $\Delta(T)-1$ precolored edges,
and all precolored edges of $M$ are pairwise nonadjacent, we can indeed
construct the required 
matchings $M_1$ and $M_2$ by for all precolored edges of $M$
greedily selecting an adjacent edge of $T_1$ (similarly for $T_2$) so that
(i)-(iii) holds.

Next, consider the edges of $M$ that are not precolored but adjacent to some
edge of $M_1$.
We assign some fixed color $c$, arbitrarily chosen from 
$\{ 1, \ldots, \Delta(T) \}$, to every uncolored edge of 
$M$ that is adjacent to an edge of $M_1$ or $M_2$. Taken together
with $\varphi$ this forms a precoloring $\varphi'$ of the graph $T \square K_2$.

We denote by $M'$ the  set containing all the remaining 
uncolored edges of $M$. Since all edges of $M_i$ and $M'$  are pairwise
nonadjacent, the set $M_i'' = M_i \cup M'$ is a matching in $G_i$. 
We set $G_i'' = G_i - M_i'' $. Then $E(G''_1) \cap M = E(G''_2) \cap M$.

By construction, each of the graphs $G_1''$ and $G''_2$ have maximum
degree $\Delta(T)$, and since both $G_1$ and $G_2$ contain at most
$\Delta(T)-1$ precolored edges under $\varphi$, respectively, 
and (i)-(iii) holds,
every component of $G''_1$ and $G''_2$ contains at most 
$\Delta(T)-1$ precolored edges under the coloring $\varphi'$.
Hence, it follows from Corollary \ref{cor1}, that there is a proper 
$\Delta(T)$-edge coloring $f_i$ of $G_i''$ that agrees with the restriction of
$\varphi'$ to $G''_i$. Note that $f_1$ and $f_2$ agree on all edges of $M$
that are in $G''_1$ (and $G''_2$).

It remains to color the edges of  $M_1'' \cup M_2''.$ 
We simply assign
color $\Delta(T) +1$ to all edges of this set. 
Taken together with $f_1$ and $f_2$, this yields a proper $(\Delta(T)+1)$-edge
coloring of $T \square K_2$ that is an extension of $\varphi$, 
because $f_1$ and $f_2$ agree on $M \cap E(G''_1)$.
This completes the proof.
\end{proof}


Let us now consider complete bipartite graphs.
Recall that the edge precoloring extension problem for the balanced complete bipartite graph $K_{n, n}$ corresponds to asking whether a partial Latin square can be completed to a Latin square. As mentioned above, motivated by
Evans' conjecture \cite{Evans},
Andersen and Hilton \cite{AndersenHilton} completely characterized 
partial Latin squares of order $n$ with $n$ nonempty cells that cannot be completed to a Latin square of order $n$. In the language of edge colorings they proved the following.

\begin{theorem}
\label{th:compbip}
\cite{AndersenHilton2} Let $n \geq 2$ be a positive integer. An  edge precoloring $\varphi$ of at most $n$ edges of $K_{n, n}$ can be extended to an $n$-edge coloring of 
$K_{n, n}$ if and only if none of the following two conditions hold.
\begin{itemize}
\item[(a)] For some uncolored edge $uv$ there are $n$ colored edges of different colors, each one incident with $u$ or $v$.
\item[(b)] For some vertex $v$ and some color $c$, the color $c$ does not appear
on any edge incident to $v$, but every uncolored edge incident with $v$
is adjacent to an edge colored $c$.
\end{itemize}
\end{theorem}
 
Using this result we now verify that Conjecture \ref{conj:general}
holds for the cartesian product $K_{n,n} \square K_2$.

\begin{theorem}
\label{compbip}
\textit{Let $n \geq 2$ be a positive integer. If $\varphi$ is an  edge precoloring of $n$ edges in $K_{n, n} \square K_2$, then $\varphi$ can be extended to a proper $(n+1)$-edge coloring of $K_{n,n} \square K_2$.}
\end{theorem}

 \begin{proof}
Without loss of generality we assume that the precoloring of $K_{n, n} \square K_2$ uses colors $1, \ldots, n$.
For $n=2$, the result follows from the above-mentioned result on hypercubes
\cite{CasselgrenMarkstromPham}. Suppose  $n > 2$ 
and let $M$ be a perfect matching in $K_{n, n} \square K_2$ such that $K_{n, n} \square K_2 - M$ is isomorphic to two copies $K^1_{n, n}$ and $K^2_{n, n}$ 
of $K_{n, n}$. We shall consider two different main cases. \\

\noindent \textbf{Case 1.} \textit{No precolored edges are in $M$:}\\

If there is at least $1$ and at most $n-1$ precolored edges in $K^1_{n, n}$,
then there is a proper $n$-edge coloring $\varphi_1$ of $K^1_{n, n}$ which is an extension of the restriction of $\varphi$ to $K^1_{n, n}$. Similarly, there is a proper $n$-edge coloring $\varphi_2$ of $K^2_{n, n}$ which is an extension of the restriction of $\varphi$ to $K^2_{n, n}$, so by coloring the edges of $M$ with color $n+1$, we obtain an extension of $\varphi$.

If $K^1_{n, n}$, on the other hand has exactly $n$ precolored edges,
then we consider a supergraph $G$ isomorphic to $K_{n+1, n+1}$ containing
$K^1_{n, n}$ as a subgraph (along with the precolored edges).
Since this graph contains at most $n$ precolored edges,
there is a proper $(n+1)$-edge coloring of $G$ that agrees with $\varphi$. 
We take the restriction of this coloring to $K^1_{n, n}$,
color $K^2_{n, n}$ correspondingly, and then color every uncolored edge of $M$ by 
the unique color not appearing at its endpoints. The obtained edge coloring
is an extension of $\varphi$.  \\

\noindent \textbf{Case 2.} \textit{$M$ contains at least one precolored edge:}\\

Let us first assume that both $K^1_{n, n}$ and $K^2_{n, n}$ contains at least one precolored edge,
Let $V^i_M$ be the set containing all vertices in $K^i_{n, n}$ that are incident with some precolored edge from $M$. 
As in the preceding proof, $v_1 \in V^1_M$ if and only 
if the corresponding vertex $v_2 \in V^2_M$. 

Proceeding along the lines in the  proof of Theorem \ref{th:tree},
we shall construct matchings $M_1 \subset E(K^1_{n, n})$ and
$M_2 \subset E(K^2_{n, n})$ such that each vertex of $V^1_M$ is incident
with a unique edge of $M_1$, and
$e_2 \in M_2$ if and only if the corresponding edge $e_1$
in $K^1_{n,n}$ is in $M_1$. 
Furthermore,
we shall require that no edge $e$ from $M_1 \cup M_2$ is adjacent to a precolored
edge of $K^1_{n,n}$ or $K^2_{n,n}$ of the same color as the precolored edge
of $M$ that e is adjacent to, and that no edge of $M_1 \cup M_2$ is precolored.
As above, since the vertex degree in $K^1_{n,n}$ and $K^2_{n,n}$ is $n$,
$K^1_{n,n}$ and $K^2_{n,n}$ are bipartite,
and $K_{n,n} \square K_2$ contains altogether $n$ precolored edges,
we can simply select the edges of $M_1$ (and $M_2$) greedily.

Now, from the restriction of $\varphi$ to $K^i_{n,n}$, we define
a new edge precoloring $\varphi_i$ 
by coloring every edge of $M_i$ by the color of the
adjacent edge of $M$ under $\varphi$.
Since $K^i_{n, n}$  contains at most $n-1$ precolored edges under $\varphi_i$, 
there is an extension $f_i$ of $\varphi_i$  using colors $1, \ldots, n$.  
Now, by recoloring all the edges in $M_1$ and $M_2$ by 
the color $n+1$ and coloring every uncolored edge of $M$ by the color not appearing at its endpoints, we obtain an extension of $\varphi$. \\

Suppose now that no precolored edges are in $K^1_{n, n}$ 
or $K^2_{n, n}$, say $K^2_{n, n}$. 
Then $K^1_{n, n}$ contains at most $n-1$ precolored edges. Our proof of this case is similar to the proof of the preceding case. As above, let $V^1_M$ be the set containing all vertices in $K^1_{n, n}$ that are incident with some precolored edge from $M$. 
Since vertices in $K^1_{n, n}$ have degree $n$, and $K^1_{n, n}$ is bipartite,
there is a
matching $M_1 \subseteq E(K^1_{n, n})$ covering $V^1_M$ satisfying analogous conditions
to the matchings constructed in the preceding case.

As before, from the restriction of $\varphi$ to $K^1_{n,n}$
we define a new precoloring $\varphi_1$ of $K^1_{n,n}$ by coloring
every edge of $M_1$ by the color of its adjacent edge in $M$.
Now, if $\varphi_1$ is extendable to a proper edge coloring of $K^1_{n, n}$ using colors $1, \ldots, n$, then we obtain an extension of $\varphi$ by recoloring all the edges in $M_1$ by color $n+1$, coloring $K^2_{n, n}$ correspondingly, and then coloring every uncolored edge of $M$ by the unique color 
from $\{1,\dots, n+1\}$ not appearing at its endpoints. So assume that there is no such extension of $\varphi_1$. By Theorem \ref{th:compbip}, this means
that  the coloring $\varphi_1$ satisfies condition (a) or (b)
of this theorem.

Suppose first that (a) holds. Then all colors $1,\dots,n$ appears on some
edge under $\varphi$, so every color appears on precisely one edge. Since $M$ contains at least one precolored edge, without loss of generality we assume that one edge $e_{M_1}$ in $M_1$ is colored with color $1$. Now define a new edge coloring $\varphi'_1$ from $\varphi_1$ by removing color $1$ from $e_{M_1}$
of  $K^1_{n, n}$. 
Since $K^1_{n, n}$ contains exactly $n-1$ $\varphi'_1$-colored edges,
there is a proper edge coloring of $K^1_{n, n}$ using colors $2, \ldots, n+1$ which is an extension of $\varphi'_1$. Now, by recoloring all the edges in
$M_1$ of $K^1_{n, n}$, distinct from $e_{M_1}$, by color $1$, coloring $K^2_{n, n}$ correspondingly, and then coloring every uncolored edge of $M$ by the unique color not appearing on any edge incident to one of its endpoints, we obtain an extension of $\varphi$.
 
Suppose now that (b) holds.
Then there is some color $c$ that appears on at least two edges 
under $\varphi$. 
If some color $c'$ is used once by $\varphi$ appears on some edge in the matching $M$, then we may remove the color $c'$ from some edge in $M_1$ 
that is colored $c'$, and proceed as in the preceding paragraph. If, on the other hand, all colors that are used once by $\varphi$ do not appear on edges of $M$, then color $c$ is the only color that appears on edges in $M$; so 
color $c$ is the only color that appears on edges in $M_1$.

Now, since $K^1_{n,n}$ contains at most $n-1$ $\varphi$-colored edges,
there is an extension of the restriction of $\varphi$ to $K^1_{n,n}$
using colors $1, \ldots, n$. 
Let $M_c$ be the set of all edges of $K^1_{n,n}$ that are colored $c$ and adjacent
to a precolored edge of $M$.
By recoloring all the edges of $M_c$ by $n+1$,  coloring $K^2_{n, n}$ correspondingly, and then coloring every uncolored edge of $M$ by the unique color not appearing at its endpoints, we obtain an extension of $\varphi$. 
\end{proof}

Finally, let us consider complete graphs. Again our confirmation of Conjecture 
\ref{conj:general} is based on a result
by Andersen and Hilton \cite{AndersenHilton2}.

\begin{theorem}
\label{th:complete}
\cite{AndersenHilton2} Let $n \geq 2$ be a positive integer.
	\begin{itemize}
	
	\item[(i)] If $\varphi$ is an edge precoloring of at most $n$ edges
	of $K_{2n}$, then $\varphi$ is extendable to a proper $(2n-1)$-edge coloring
	unless the precolored edges form a matching, where $n-1$ edges are colored
	by a fixed color $c$, and one edge is colored by some color $c'\neq c$.
	
	\item[(ii)] If $\varphi$ is an edge precoloring of at most $n+1$ edges
	of $K_{2n-1}$, then $\varphi$ is extendable to a proper $(2n-1)$-edge coloring
	unless the precolored edges form a set of $n-2$ independet edges colored
	by a fixed color $c$, and a triangle, disjoint from the independent edges,
	the edges of which are  colored by three different colors
	that are distinct from $c$.

\end{itemize}
\end{theorem}
	In particular, this implies that every partial edge coloring
	of at most $n-1$ edges is extendable to a proper $(2n-1)$-edge coloring
	of $K_{2n}$, and similarly, every partial edge coloring of at most
	$n$ edges of $K_{2n-1}$ is extendable to a proper $(2n-1)$-edge coloring.
	
	The following establishes that Conjecture \ref{conj:general} holds for complete
	graphs.

\begin{theorem}
\label{prop:complete}
Let $n \geq 2$ be a positive integer.
	\begin{itemize}
	
	\item[(i)] If $\varphi$ is an edge precoloring of at most $n$ edges
	of $K_{2n} \square K_2$, then $\varphi$ is extendable to a proper $2n$-edge coloring
	of $K_{2n} \square K_2$.
	
	\item[(ii)] If $\varphi$ is an edge precoloring of at most $n+1$ edges
	of $K_{2n-1} \square K_2$, then $\varphi$ is extendable to a 
	proper $2n$-edge coloring of $K_{2n-1} \square K_2$.
	
	\end{itemize}
\end{theorem}

We note that the number of colors used in part (ii) is best possible, since
there are partial edge colorings of just two edges in $K_{2n-1} \square K_2$
that are not extendable to proper $(2n-1)$-colorings.

Before, we prove the general case of Theorem \ref{prop:complete}, we separately consider 
the case of $K_5 \square K_2$.
We first note the following lemma, which is easily proved using the fact
that $K_{2n-1} \square K_2$ is Class 1.
	
	\begin{lemma}
	\label{cl:alledges}
		If $\varphi$ is an edge precoloring of $K_{2n-1} \square K_2$ with $n+1$ precolored edges, 
		where at least $n$ edges have the same color, then
		$\varphi$ is extendable.
	\end{lemma}

In the following, we shall say that a matching {\em covers} a set $S$ of edges,
if every edge of $S$ is adjacent to an en edge of the matching.

	\begin{lemma}
	\label{cl:K5}
		If $\varphi$ is a partial edge coloring of at most $4$ edges of $K_5 \square K_2$,
		then $\varphi$ is extendable to a proper $6$-edge coloring.
	\end{lemma}
	\begin{proof}
		Let $M$ be a perfect matching in 
		$K_5 \square K_2$ such that $K_5 \square K_2 - M$ is 
		isomorphic to two copies $K_5^1$ and $K_5^2$ of $K_5$.
		If no precolored edge is in $M$, then we may proceed as in the proof of Theorem
		\ref{compbip}. Thus we assume that at least one edge of $M$ is precolored. 
		
		We shall consider many different cases. By Lemma \ref{cl:alledges}, 
		we may assume that at least two different colors are used in the precoloring
		$\varphi$.

		\bigskip
		
		{\bf Case 1.} {\it All precolored edges are in $E(K^1_{5}) \cup M$:}
		
		Suppose first that exactly one color appears on the precolored edges of $M$, 
		color $1$ say. 
		Then we consider the precoloring of $K^1_{5}$ obtained from $\varphi$
		by removing the color $1$ from any edge $K^1_{5}$ that is precolored $1$.
		By Theorem \ref{th:complete}, 
		this precoloring is extendable to a proper edge coloring
		of $K^1_{5}$ using colors $2,\dots,6$. By recoloring any edge of $K^1_{5}$ 
		that is $\varphi$-precolored $1$ by the color $1$, 
		coloring $K^2_{5}$ correspondingly, and then coloring
		every uncolored edge of $M$ by an appropriate color missing at 
		its endpoints we obtain an extension of
		$\varphi$.
		
		Suppose now that at least two colors appears on the precolored edges of $M$.
		We consider some different subcases.
		
		\bigskip
		{\bf Case 1.1.} {\it Exactly two edges of $M$ are precolored:}
		
		If there is at least one color $c$ used on the precolored edges of $K^1_{5}$
		that neither appears on an edge of $M$, nor is the edge precolored $c$
		adjacent to a precolored edge of $M$,
		then there is a matching $M_1 \subseteq E(K^1_{5})$ of uncolored edges
		satisfying the following
		\begin{itemize}
		
			\item every edge of $M_1$ is adjacent to exactly one precolored edge of $M$,
			
			\item every precolored edge of $M$ is adjacent to an edge of $M_1$,
			
			\item no edge of $M_1$ is adjacent to two precolored edges $e_1 \in E(K^1_{5})$
			and $e_2 \in M$ that have the same color under $\varphi$.
		
		\end{itemize}
			We call such a matching a {\em $\varphi$-good} matching. Moreover, 
			it is easy to see
			that we can pick this matching such that if we color the edges of $M_1$ by the
			color of the precolored adjacent edges in $M$, then this coloring along with
			the restriction of $\varphi$ to $K^1_{5}$ does not satisfy the 
			condition in Theorem
			\ref{th:complete}. Hence, the obtained $5$-edge precoloring of $K^1_{5}$
			is extendable. By recoloring every edge of $M_1$
			by the color $6$, coloring $K_2$ correspondingly, and then coloring the edges
			of $M$ by an appropriate color not appearing at its endpoints, 
			we obtain an extension
			of $\varphi$.
			
			Suppose now that there is no $\varphi$-good matching, but there is a color $c_1$
			that appears in $K^1_5$ but not in $M$. 
			Since there is no $\varphi$-good matching,
			the edge $e_1$ colored $c_1$
			is adjacent to an edge $e$ of $M$ that is colored $c_2 \neq c_1$, but not to
			the other precolored edge of $M$. Moreover, there is another
			edge $e'_1$ in $K^1_5$ colored $c_2$. 
			Now, it is easy to see that this implies that there is a proper edge
			coloring of $K_5 \square K_2$ with colors $\{1,\dots,6\} \setminus \{c_2\}$ 
			that agrees with all precolored edges that are not precolored $c_2$.
			Hence, $\varphi$ is extendable.
			
			Next, assume that the same two colors, say $1$ and $2$, 
			are used both on the precolored
			edges of $M$ and on $K^1_{5}$. 
			Then there is at most one vertex of degree $4$
			in the graph obtained from $K^1_{5}$ by removing all precolored edges.
			We properly color the uncolored edges of $K^1_{5}$
			using colors $3,4,5,6$, color $K^2_{5}$ correspondingly, 
			and then proceed as before.
			
			\bigskip
		
		\bigskip
		{\bf Case 1.2.} {\it Exactly three edges of $M$ are precolored:}
		
		Let $e_1$ be 
		the precolored
		edge of $K^1_{5}$ and assume first that only two colors, say $1$ and $2$, 
		are used in the 
		precoloring $\varphi$.
		By Lemma \ref{cl:alledges}, we may assume that $\varphi(e_1) =1$ and exactly one edge
		of $M$ is colored $1$. Then we can pick an uncolored edge $e'_1$ of $K^1_5$, 
		that is not
		adjacent to any edge $\varphi$-colored $1$. We color $K^1_5 - \{e_1,e'_1\}$
		properly by colors $3,4,5,6$, color $e'_1$ by $1$ and proceed as before.
		
		Suppose now that three colors appear in the precoloring $\varphi$. If the
		color of the precolored edge in $K^1_{5}$ does not appear on an edge of
		$M$, then there is some color $c$ that appears on two edges in $M$, and a color
		$c'$ that only appears on one edge of $M$.
		This implies that there is a matching $M_1 \subseteq E(K^1_{5})$ of
		uncolored edges
		satisfying the following
		\begin{itemize}
		
			\item every edge of $M_1$ is adjacent to exactly one precolored edge of $M$,
			
			\item every precolored edge of $M$ colored $c$ is adjacent to an edge of $M_1$,
			
			\item the edge precolored $c'$ is not adjacent to an edge of $M_1$.
		
		\end{itemize}
		Consider the precoloring obtained from the restriction of $\varphi$ to $K^1_5$
		by in addition coloring every edge of $M_1$ by the 
		color $c$.
		It follows from Theorem \ref{th:complete} that there is an 
		extension of this coloring
		to $K^1_{5}$ using colors $\{1,2,3,4,5,6\} \setminus \{c'\}$. 
		Now, we obtain an extension
		of $\varphi$ by coloring the edges of $M_1$ by the color $c'$,  coloring $K^2_{5}$
		correspondingly, and finally coloring the edges of $M$ appropriately.
		
		If, on the other hand, there is a color $c$ which appears 
		both in $K^1_{5}$ and in $M$,
		then we proceed as follows.
		We pick a proper edge coloring $f$ of $K^1_{5}$ using 
		colors $\{1,2,3,4,5,6\}\setminus \{c\}$,
		so that for every vertex $v \in V(K^1_{5})$, 
		if there is a precolored edge of $M$ incident with
		$v$, then the color of the edge of $M$ does not appear at $v$ under $f$.
		Next, we recolor the edges of $K^1_5$ $\varphi$-precolored $c$ 
		by the color $c$, and obtain an extension
		of $\varphi$ as before.

		The case when all four colors appear in $\varphi$ 
		can be dealt with using a similar argument
		as in the preceding paragraph.

		\bigskip
		
		\bigskip
		{\bf Case 1.3.} {\it Exactly four edges of $M$ are precolored:}

		Suppose now that all four precolored edges are in $M$.
		If at most three colors appear in the precoloring $\varphi$, then there
		is a color $c$ that appears on at least two edges. 
		If only two colors appear in $\varphi$, then by Lemma \ref{cl:alledges},
		we may assume that both colors appear on two edges. Thus
		there is a matching $M_1$
		in $K^1_5$ covering all precolored edges of $M$ and such that no
		edge in $M_1$ is adjacent to two edges precolored with different colors.
		Thus, $\varphi$ is extendable, as before.
		If, on the other hand three colors appear in $\varphi$, then we can pick
		a proper edge coloring $f$ of $K^1_5$ using colors $\{1,\dots, 6\} \setminus \{c\}$,
		such that no precolored edge of $M$ is adjacent to an edge of the same color
		under $f$. Hence, $\varphi$ is extendable. Note that we can use a similar
		argument if all of the colors $1,2,3,4$ appear on some edge under $\varphi$.

		\bigskip
		
		Note that in all cases above, $K^1_{5}$ is first colored, and then $K^2_{5}$
		is colored correspondingly. We shall use this property 
		when we consider the next case.
		
		\bigskip
		
		{\bf Case 2.} {\it Both $K^1_{5}$ and $K^2_{5}$ contains at least one precolored
		edge:}
		
		The condition imlies that $M$ contains one or two precolored edges.
		Assume first that $M$ contains only one precolored edge, and so we may assume that
		$K^1_{5}$ contains two precolored edges, and $K^2_{5}$ one.
		Let $e_1$ and $e'_1$ be the precolored edges of $K^1_{5}$, $e_2$ and $e'_2$
		the corresponding edges of $K^2_{5}$ respectively,
		and let $e''_2$ be the precolored edge of $K^2_{5}$, and $e''_1$ the corresponding
		edge of $K^1_{5}$. Furthermore, let $e$ be the precolored edge of $M$.

		Consider the restriction of $\varphi$ to $K^1_{5}$.
		If we can assign the color $\varphi(e''_2)$ to $e''_1$ so that the 
		resulting coloring
		$\varphi_1$ of $K^1_{5}$ is proper,
		then we may proceed as in Case 1 (since in that case $K^2_{5}$ 
		is always colored correspondingly).
		Thus we may assume that either
		\begin{itemize}
		
			\item[(a)] $e''_1 \in \{e_1,e'_1\}$, or 
			
			\item[(b)] $e''_1$ is adjacent to one of
			the edges in $\{e_1,e'_1\}$ and $e''_2$ has the same color as 
			one adjacent edge in $\{e_1,e'_1\}$.
		\end{itemize}
		
		By Lemma \ref{cl:alledges}, we may further assume that no color 
		appears on three edges, and thus
		at most two edges are precolored by the same color. Then, unless
		$e$ is adjacent to two precolored edges and there is another
		edge, $e_1$ say, precolored $\varphi(e)$, that is disjoint from all these three edges,
		there is an uncolored
		edge $e^{(3)}_1$ in $K^1_5$ that is adjacent to $e$ but not adjacent to a 
		precolored edge of $K^1_5$ colored $\varphi(e)$, 
		and, similarly for the corresponding edge $e^{(3)}_2$ of $K^2_5$.
		From the restriction of $\varphi$ to $K^1_5$ and $K^2_5$ we obtain new 
		$5$-edge precolorings
		by coloring the edges 
		$e^{(3)}_1$ and $e^{(3)}_2$  by the color $\varphi(e)$. 
		Now by Theorem \ref{th:complete}
		these precolorings are extendable, and we may finish 
		the argument by proceeeding as above.
		
		Suppose now that $e$ is adjacent to two precolored edges and there is
		one additional edge precolored $\varphi(e)$ in $E(K^1_5) \cup E(K^2_5)$.
		It is not hard to see that this implies that there are uncolored corresponding edges
		$e^{(4)}_1 \in E(K^1_{5})$ and $e^{(4)}_1 \in E(K^2_{5})$
		that are not adjacent to any edges precolored $\varphi(e)$. Hence, by coloring
		these edges $\varphi(e)$, and also, $e_2$ by the color $\varphi(e)$,
		from the restrictions of $\varphi$ to $K^1_5$ and $K^2_5$, respectively,
		we obtain extendable precolorings. Moreover, by construction, 
		the extensions of these colorings
		satisfy that no edge colored $c$ is adjacent to $e$. Hence, $\varphi$
		is extendable.

		\bigskip
		
		Let us now assume that $M$ contains two precolored edges. So $K^1_5$ and
		$K^2_5$ both contains precisely one precolored edge, $e_1$ and $e'_2$, respectively.
		Denote the corresponding edges of $K^2_5$ and $K^1_5$ by $e_2$ and 
		$e'_1$, respectively.
		As above, it follows that either 
		\begin{itemize}
		
		\item [(a)] $e_1 = e'_1$ and $\varphi(e_1) \neq \varphi(e'_2)$, or
		
		\item [(b)] $e_1$ and $e_2$ are adjacent and $\varphi(e_1) = \varphi(e'_2)$.
		
		\end{itemize}
		
		Suppose first that (a) holds.
		If both colors in $\{\varphi(e_1), \varphi(e'_2)\}$, say $1$ and $2$, 
		appear on the precolored
		edges of $M$, 
		then we pick corresponding uncolored edges $e''_1 \in E(K^1_5)$
		and $e''_2 \in E(K^2_5)$ that are not adjacent to $e_1$ or $e_2$ and
		color $e''_1$ and $e''_2$ by $1$ or $2$ so that the obtained coloring is proper. 
		Thereafter, we
		color $K^1_5-\{e'_1, e''_1\}$
		and $K^2_5-\{e'_2, e''_2\}$ properly using 
		colors $3,4,5,6$, and proceed as above.

		If  both colors in $\{\varphi(e_1), \varphi(e_2)\}$ do not appear on the precolored
		edges of $M$, then there are matchings $M_1 \subseteq E(K^1_5)$ 
		and $M_2 \subseteq E(K^2_5)$ of corresponding uncolored edges such that
		\begin{itemize}

		\item every precolored edge of $M$ is adjacent to exactly one edge of $M_i$,
		
		\item every edge of $M_i$ is adjacent to a precolored edge of $M$,

		\item no edge of $M_i$ is adjacent to two precolored edges of the same color. 
		
		\end{itemize}
		Now, consider the restriction of $\varphi$ to $K^1_5$ and $K^2_5$, respectively.
		By, in addition, coloring the edges of $M_1$ and $M_2$ by the color of an adjacent
		precolored edge of $M$ we obtain extendable $5$-edge precolorings of 
		$K^1_5$ and $K^2_5$, respectively. 
		Given extensions of these precolorings, we may recolor the edges of
		$M_1$ and $M_2$ by color $6$, and then color the edges of 
		$M$ appropriately to obtain
		an extension of $\varphi$ as before.
		
		Suppose now that (b) holds. 
		Since $\varphi(e_1) = \varphi(e'_2)$, and any color appears
		on at most two edges under $\varphi$, there are matchings 
		$M_1$ and $M_2$ as described
		in the preceding paragraph.
		Thus, we proceed similarly, and
		this completes the proof of the lemma.
	\end{proof}

\begin{proof}[Proof of Theorem \ref{prop:complete}.]
We first prove part (i).
	Denote by $M$ the matching of $K_{2n} \square K_2$
	such that $K_{2n} \square K_2 - M$ is isomorphic to two copies
	$K^1_{2n}$ and $K^2_{2n}$. The cases when no precolored edges
	are in $M$, can be handled as in the proof of Theorem \ref{compbip},
	so we omit the details here.
	
	In the case when $M$ contains at least one
	precolored edge, then we may select a matching $M_1$ in $K^1_{2n}$
	as in the proof of Theorem \ref{compbip} (and possibly also a matching $M_2$
	of corresponding edges in $K^2_{2n}$). The only essential difference in the
	argument is that since $K_{2n}$ contains triangles, we can only ensure
	that $n$ such edges forming the matching $M_1$ can be selected greedily,
	although $K_{2n}$ has vertex degree $2n-1$. Nevertheless, since
	$K_{2n} \square K_2$ contains at most $n$ precolored edges, this suffices
	for our purposes. Apart from this difference, the argument is very similar to the
	one in the proof of Theorem \ref{compbip}, so we omit the details.
	
	\bigskip
	
	Let us now prove part (ii).
	Denote by $M$ the matching of $K_{2n-1} \square K_2$
	such that $K_{2n-1} \square K_2 - M$ is isomorphic to two copies
	$K^1_{2n-1}$ and $K^2_{2n-1}$ of $K_{2n-1}$.

	The case of $K_3$ follows from the result on odd cycles proved in Section
	4, and the case of $K_5$
	is dealt with by the above lemma, so let us now consider the case of 
	$K_{2n-1} \square K_2$ for $n \geq 4$.
	
	We shall consider a number of different cases.
	In many of these cases we shall use strategies which are similar to the 
	ones used in the proof
	of Theorem \ref{compbip} and/or Lemma \ref{cl:K5}. 
	Therefore we shall be content with sketching the arguments, 
	when they are similar
	to ones that have been described above.

	As before, the case when no precolored edge is in $M$ can be dealt with as in the
	proof of Theorem \ref{compbip}, so in the following
	we assume that at least one precolored edge is
	contained in $M$.
	The rest of the proof breaks into the following cases:
	
	\begin{itemize}
	
	\item[(1)] Only one color appear on the precolored edges in $M$.
	
	\item[(2)] At least two colors appear on the precolored edges in $M$, but
	at most one color appears on the edges in $E(K^1_{2n-1}) \cup E(K^2_{2n-1})$.
	
	\item[(3)] At least two colors appear on the precolored edges of $M$
	and two colors appear on the precolored edges in 
	$E(K^1_{2n-1}) \cup E(K^2_{2n-1})$.
	
	\end{itemize}
	
	\bigskip

	{\bf Case 1.} {\it Only one color appears on the precolored edges in $M$:}

	Assume, for simplicity, that color $1$ appears on the edges of $M$.
	If all precolored edges are in $M \cup E(K^1_{2n-1})$, then consider
	the precoloring obtained from the restriction of $\varphi$ to $K^1_{2n-1}$
	by removing the color $1$ from all edges $\varphi$-colored $1$. 
	By Theorem \ref{th:complete}, this precoloring
	is extendable to a proper edge coloring using colors $2,\dots,2n$, 
	and we obtain an extension of $\varphi$ by recoloring 
	the edges of $K^1_{2n-1}$ that are $\varphi$-colored $1$ by the color $1$,
	coloring $K^2_{2n-1}$ correspondingly and then coloring every edge of $M$
	by the color $1$ or $2n$.
	
	Suppose now that both $K^1_{2n-1}$ and $K^2_{2n-1}$ contains at least one precolored edge.
	By Lemma \ref{cl:alledges}, we may assume that there are at least two
	edges in $E(K^1_{2n-1}) \cup E(K^2_{2n-1})$ precolored by a color distinct from $1$.
	We select a matching $M_1$ of edges in $K^1_{2n-1}$ of 
	maximum size, so that $M_1$ and the set $M_2$
	of corresponding edges in $K^2_{2n-1}$ satisfy the following:
	\begin{itemize}
	
		\item every precolored edge of $M_i$ is adjacent to at least one
		precolored edge of $M$,
		
		\item no edge of $M_i$ is adjacent to an edge of $E(K^1_{2n-1}) \cup E(K^2_{2n-1})$
		that is precolored $1$,
		
		\item no edge of $M_i$ is precolored.
	
	\end{itemize}
	If every precolored edge of $M$ is adjacent to an edge of $M_1$, then we consider the
	precolorings obtained from the restrictions of $\varphi$ to $K^1_{2n-1}$ 
	and $K^2_{2n-1}$,
	respectively, by in addition coloring all edges of $M_1$ and $M_2$ by the color $1$,
	and then proceed as before.
	
	Suppose now that some precolored edge of $M$ is not adjacent to an edge of $M_1$.
	Since at most $n-2$ edges of $E(K^1_{2n-1}) \cup E(K^2_{2n-1})$ are precolored $1$,
	edges of $M_1$ may be adjacent to several precolored edges of $M$,
	and $M_1$ is maximum with respect to the aforementioned properties, 
	it follows that it must be the case that only one 
	edge $u_1u_2$ of $M$ is precolored (where $u_i \in V(K_{2n-1}^i$),
	there is a matching $M'$ of $n-2$ edges in 
	$K^1_{2n-1}$
	such that every edge of $M'$ is either precolored $1$, or the corresponding
	edge of $K^2_{2n-1}$ is precolored $1$. Moreover, $u_1$ is 
	incident with two edges $e_1$
	and $e_2$ that are independent from $M'$ and satisfying that 
	$e_i$ or the corresponding
	edge of $K^2_{2n-1}$ is precolored by a color distinct from $1$.
	Now, from the restriction of $\varphi$ to $K^1_{2n-1}$ we 
	define a new precoloring $\varphi_1$
	by coloring all edges of $M'$ by the color $1$, and, in addition, coloring the unique
	edge of $K^1_{2n-1}$ that is adjacent to both $e_1$ and $e_2$ by the color $1$. We define
	an analogous precoloring $\varphi_2$ of $K^2_{2n-1}$.
	
	By Theorem \ref{th:complete}, both $\varphi_1$ and $\varphi_2$ are extendable to
	proper $(2n-1)$-edge colorings. Moreover, it is easy to see that neither
	$u_1$ nor $u_2$ is
	incident to an edge colored $1$ in these extensions. Consequently,
	$\varphi$ is extendable.

	\bigskip
	
	{\bf Case 2.}
	{\it At least two colors appear on the precolored edges in $M$, but
	at most one color appears on the edges in $E(K^1_{2n-1}) \cup E(K^2_{2n-1})$:}
	
	We first consider the case when all precolored edges are in  $M \cup E(K^1_{2n-1})$.
	
	Suppose first that all  precolored edges are contained in $M$. If every color
	appears on at most one edge in $M$, then $\varphi$ is extendable, 
	because we can choose
	a proper $(2n-1)$-edge coloring $f$ of $K^1_{2n-1}$ 
	so that for every vertex $v \in V(K_{2n-1})$,
	if the edge of $M$ incident with $v$ is colored $i \in \{1,\dots, 2n-1\}$, 
	then no edge incident with $v$ is colored $i$ under  $f$. 
	A similar argument applies if one color appears on 
	at least two edges in $M$
	and all other colors appear on at most one edge. 
	
	Assume now that there are two colors $c_1$ and $c_2$ 
	that both appear on at least two edges.
	Then there is matching $M_1$ of uncolored edges in $K^1_{2n-1}$ 
	covering all vertices incident with precolored
	edges, and so that no edge of $M$ is adjacent to 
	two precolored edges of two different colors.
	By proceeding as before, it is now straightforward that $\varphi$ is extendable.
	
	\bigskip
	
	Suppose now that $E(K^1_{2n-1})$ contains at least one precolored edge, 
	precolored $c$ say.
	If all colors on the edges of $M$ are distinct, except that several
	edges of $M$ may be colored $c$, then there is a proper edge coloring $f$ of $K^1_{2n-1}$
	using colors $\{1,\dots, 2n\} \setminus \{c\}$, such that no precolored edge of $M$
	is adjacent to an edge of the same color under $f$. Hence, $\varphi$ is extendable.
	
	Assume now instead that some color $c_1 \neq c$ appears on at least two
	edges of $M$ and every other color used on an edge of $M$
	appear on exactly one edge of $M$. If 
	at most one edge of $K^1_{2n-1}$ is precolored $c$,
	then we
	proceed as before and pick a proper edge coloring $f$ of $K^1_{2n-1}$
	using colors $\{1,\dots, 2n\} \setminus \{c_1\}$
	that agrees with the restriction of $\varphi$ to $K^1_{2n-1}$ and
	where no precolored edge of $M$ is adjacent to an edge of the same color under
	$f$. If instead color $c$ is used on at least two edges of $K^1_{2n-1}$,
	then it is easy to see that there is a matching $M_1$ of uncolored edges in 
	$K^1_{2n-1}$ covering all 
	precolored edges of $M$, and such that no edge of $M_1$ is 
	adjacent to two precolored edges of $M$ that are colored by distinct colors
	or one precolored edge of $M$ and one from $K^1_{2n-1}$ that have the same color.
	Hence, $\varphi$ is extendable.

	Finally, let us assume that there are two colors $c_1,c_2 \neq c$
	that both appear on at least two edges of $M$. 
	Then there is
	a matching $M_1$ of uncolored edges in $K^1_{2n-1}$ covering all 
	precolored edges of $M$, such that no edge of $M_1$ is 
	adjacent to two precolored edges of $M$ that are colored by distinct colors
	or one precolored edge of $M$ and one from $K^1_{2n-1}$ that have the same color;
	arguing as in the preceding paragraph, it is easily verified that 
	our assumption on the colors on edges of $M$ implies that
	there is such a matching.	
	Since $M_1$ covers all precolored edges of $M$, we may proceed as before to obtain
	an extension of $\varphi$.

	\bigskip
	
	Let us now consider the case when both $K^1_{2n-1}$ and $K^2_{2n-1}$ 
	contains at least one precolored edge, precolored $c$ say.
	If all the colors on edges of $M$ are distinct, except that several
	edges of $M$ may be colored $c$, then $\varphi$ is extendable as in the preceding case.
	Consequently, suppose that there is a color $c_1 \neq c$ that appears on at least two
	edges $u_1u_2$ and $v_1v_2$ of $M$, where $u_i,v_i \in V(K^i_{2n-1})$. 
	
	We may assume that at least one precolored edge of $K^2_{2n-1}$ satisfies that
	the corresponding edge of $K^1_{2n-1}$ is adjacent to a precolored edge of
	$K^1_{2n-1}$, since otherwise we can define a precoloring of $K^1_{2n-1}$
	by coloring every edge $e$ that is precolored $c$, or satisfying that
	the corresponding edge of $K^2_{2n-1}$ is precolored $c$, by the color $c$,
	and then proceed as in the case when only
	edges of $K^1_{2n-1}$ and $M$ are precolored.
	This assumption implies that there are matchings $M_1$ and $M_2$
	of corresponding uncolored edges of $K^1_{2n-1}$ and $K^2_{2n-1}$, respectively,
	(containing $u_1v_1$ and $u_2v_2$, respectively, if these edges are uncolored)
	such that
	\begin{itemize}
		
		\item $M_i$ covers the precolored edges of $M$,
		
		\item no edge of $M_i$ adjacent to two precolored edges of $M$ of different colors,
		
		\item no edge of $M_i$ is adjacent to an edge of $K^i_{2n-1}$ precolored $c$ and
		an edge of $M$ precolored $c$.	
	\end{itemize}
	As before, this implies that $\varphi$ is extendable.

	\bigskip
	
	{\bf Case 3.}
	{\it At least two colors appear on the precolored edges of $M$
	and at least two colors appear on the precolored edges 
	in $E(K^1_{2n-1}) \cup E(K^2_{2n-1})$:}

	Suppose first that all precolored edges lie in $M \cup E(K^1_{2n-1})$.
	If there is a matching $M_1$ of uncolored edges in $K^1_{2n-1}$ covering
	all precolored edges of $M$, and such that no edge of $M_1$ is adjacent
	to two edges of $M$ of different colors, or a precolored edge of $M$
	and one of $K^1_{2n-1}$ of the same color, then $\varphi$ is extendable.
	Such a matching is called a {\em good matching}.
	
	On other hand, if there is no good matching, then at most two edges
	of $M$ are precolored. This is so, because otherwise we could
	select three edges for $M_1$, each of which is adjacent to at least 
	one precolored edge of $K^1_{2n-1}$, and thereafter select
	the rest of the edges of $M_1$ greedily.
	
	Now, since at least two colors appear on the edges in $K^1_{2n-1}$,
	if there is no good matching, then one precolored edge of $M$,
	colored $c_1$ say,
	must be adjacent to an edge $e'$ precolored $c_2 \in \{1,\dots, 2n-1\}$, and there
	is a matching $M'$ of 
	$n-2$ edges in $K^1_{2n-1}$, disjoint from $e'$, that is precolored $c_1$. 
	It is straightforward
	to verify that this precoloring is extendable, e.g. by first taking an extension
	of the edges colored by colors distinct from $c_1$ using colors 
	$\{1,\dots, 2n\} \setminus \{c_1\}$.
	
	\bigskip
	
		Let us now consider the case when both $K^1_{2n-1}$ and $K^2_{2n-1}$ 
	contains at least one precolored edge.
	Again, the idea is to select matchings $M_1$ and $M_2$ of uncolored
	and corresponding edges of $K^1_{2n-1}$ and $K^2_{2n-1}$, respectively, 
	satisfying the following conditions:
	\begin{itemize}
		
		\item $M_i$ covers the precolored edges of $M$,
		
		\item no edge of $M_i$ adjacent to two precolored edges of $M$ of different colors,
		
		\item no edge of $M_i$ is adjacent to an edge of $K^i_{2n-1}$ 
		and an edge of $M$ that are precolored by the same color.
	\end{itemize}
	If there are such matchings $M_1$ and $M_2$, then we can use them
	for finding an extension of $\varphi$ as before.
	
	Suppose now that there are no such matchings. It follows, as above, that then
	at most two edges of $M$ are precolored, by colors $c_1$ and $c_2$ say.
	Now, again as in the preceding case, one of the two precolored edges of $M$,
	the one precolored $c_1$ say, is adjacent to an edge $e'$ colored 
	$c_3 \in \{1,\dots, 2n-1\}$ 
	in $K^1_{2n-1}$
	or $K^2_{2n-1}$, say $K^1_{2n-1}$, and there is matching $M'$ in $K^1_{2n-1}$
	of $n-2$ edges, disjoint from
	$e'$,
	such that every edge in $M'$, or the corresponding edge in $K^2_{2n-1}$,
	is precolored $c_1$. Thus we can color all edges of this matching by the
	color $c_1$ and proceed as before. This completes the proof of the theorem.
\end{proof}

\section{Cartesian products with general graphs}

We have not been able to confirm Conjecture \ref{conj:general}
in the general case, but we can prove it for the case of regular triangle-free
graphs when the precolored edges are independent.

\begin{theorem}
\label{th:main}
	If $G$ is a triangle-free regular graph where every precoloring 
	of at most $k < \Delta(G)$
	independent edges
	are extendable to a $\chi'(G)$-edge coloring, then every precoloring of
	at most $k+1$ independent edges in 
	$G \square K_2$ is extendable
	to a $(\chi'(G)+1)$-edge coloring.
\end{theorem}
\begin{proof}
	Without loss of generality, we assume that 
	$k+1$ edges of $G \square K_2$
	are precolored. We denote this precoloring by $\varphi$, 
	by $G_1$ and $G_2$ the copies of $G$ in $G \square K_2$, respectively,
	and by $M$ the perfect matching between $G_1$ and $G_2$.
	We shall consider some different cases.
	
	\bigskip
	{\bf Case 1.} {\em All precolored edges appear in $E(G_1) \cup E(G_2)$:}
		
		If both $G_1$ and $G_2$ contains at least one precolored edge,
		then by assumption, both the restriction of $\varphi$ to $G_1$
		and to $G_2$ are extendable to proper $\chi'(G)$-edge colorings
		of $G_1$ and $G_2$, respectively. By coloring all edges of $M$
		by color $\chi'(G)+1$, we obtain an extension of $\varphi$.
		
		Suppose now that all precolored edges appear in one of $G_1$ and 
		$G_2$, $G_1$ say. Without loss of generality, assume that
		color $1$ appears on at least one edge in $G_1$.
		By removing the color $1$ from any edge precolored $1$,
		we obtain a partial edge coloring of $G_1$ that is extendable
		to a proper edge coloring of $G_1$ using colors 
		$2,\dots, \chi'(G)+1$. By recoloring all edges precolored $1$
		under $\varphi$ by the color $1$, coloring $G_2$ correspondingly
		and then coloring all edges of $M$ by a color missing at its
		endpoints, we obtain an extension of $\varphi$.
	
	\bigskip
	{\bf Case 2.} {\em All precolored edges are in $E(G_1) \cup M$:}
	
	We assume that at least one precolored edge is in $M$.
	Let $E_M$ be the set of all precolored edges in $M$.
	We shall select a matching $M_1$ of $|E_M|$ uncolored
	edges in $G_1$ satisfying the following:
	
	\begin{itemize}
	
		\item[(i)] every edge of $E_M$ is adjacent to exactly
		one edge of $M_1$;
		
		\item[(ii)] every edge of $M_1$ is adjacent to exactly one precolored
		edge of $G$.
	
	\end{itemize}

	Since $G_1$ is regular and triangle-free, $G \square K_2$ contains
	at most $\Delta(G_1)$ precolored edges and all those precolored
	edges are independent, it is straightforward to verify that there
	is a set $M_1 \subseteq E(G_1)$ satisfying (i)-(ii); indeed,
	since every vertex of $G_1$ has degree $\Delta(G_1)$,
	we can simply select edges adjacent to the precolored edges of
	$M$ greedily.
	
	Now, by coloring all edges of $M_1$ by the color of the adjacent edge
	in $M$, and taking this coloring together with the restriction
	of $\varphi$ to $G_1$, we obtain a precoloring $\varphi_1$
	of $G_1$ with $k+1$ precolored independent edges.
	Without loss of generality we assume that some edge is colored $1$
	under $\varphi_1$. By removing the color $1$ from every such
	edge $\varphi_1$-precolored $1$, we obtain a precoloring that is extendable to
	a proper edge coloring of $G_1$ using colors $2, \dots, \chi'(G)+1$.
	
	Next, for every edge of $G_1$ that is
	$\varphi$-precolored $1$, we recolor this edge by $1$. Similarly, for every
	$\varphi_1$-precolored edge $e$ of $M_1$ such that $\varphi_1(e) \neq 1$,
	we recolor $e$ by the color $1$ and the adjacent edge in $M$ by the
	color $\varphi_1(e)$. Finally, we recolor any edge of $M$ that is 
	$\varphi$-precolored $1$ by the color $1$.
	Since all $\varphi_1$-precolored edges are independent, the resulting
	partial edge coloring of $G$ is proper. 
	By coloring $G_2$ correspondingly
	and coloring all uncolored edges of $M$ by a color
	missing at its endpoints, we obtain a proper edge coloring that is an extension of
	$\varphi$.

	\bigskip
	{\bf Case 3.} {\it $E(G_1), E(G_2)$ and $M$ contains at least one
	precolored edge each:}
	
	Let $\varphi_1$ and $\varphi_2$ be the restrictions
	of $\varphi$ to $G_1$ and $G_2$, respectively.
	As in the preceding case we shall select a matching of $|E_M|$ uncolored edges 
	$M_1 \subseteq E(G_1)$ and a matching 
	$M_2 \subseteq E(G_2)$ of uncolored corresponding edges satisfying the following:
	
	\begin{itemize}
	
		\item[(i)] every edge of $E_M$ is adjacent to exactly
		one edge of $M_i$;
		
		\item[(ii)] every edge of $M_1 \cup M_2$ is adjacent to exactly one 
		precolored edge of $G$.
	
	\end{itemize}
	
	The existence of such sets $M_1$ and $M_2$ follows as in Case 2,
	since $G_1$ and $G_2$ are $\Delta(G)$-regular and $G \square K_2$
	contains altogether at most $\Delta(G)$ precolored edges.
	
	Now, consider the edge precolorings obtained from $\varphi_1$ and
	$\varphi_2$, respectively, by coloring every edge of $M_1$ and $M_2$
	by the color of the adjacent precolored edge of $M$.
	Since $G_1$ and $G_2$ both contains at least one $\varphi$-precolored edge,
	the obtained precolorings $\varphi'_1$ and $\varphi'_2$, respectively,
	are by assumption extendable to proper $\chi'(G)$-edge colorings.
	Now, we recolor every edge of $M_1 \cup M_2$ by the color $\chi'(G)+1$,
	and then color every $\varphi$-colored edge of $M$ by its 
	color under $\varphi$, and by any color not appearing at its
	endpoints if it is not $\varphi$-colored.
	Since $M_i$ is a matching, the resulting coloring is proper, and
	thus also an extension of $\varphi$.

	\bigskip
	{\bf Case 4.} {\it All precolored edges are in $M$:}
	
	Without loss of generality, assume that at least one edge of $M$
	is colored $1$. Let $E_1$ be the set of precolored edges of $M$
	that are not colored $1$.
	Since $G$ is $\Delta(G)$-regular and triangle-free, 
	there is a matching $M_1 \subseteq E(G_1)$
	of $|E_1|$ edges such that
	every edge of $E_1$ is adjacent to exactly
	one edge of $M$, and every edge of $M_1$ is adjacent to 
	at most one edge of $E_1$ and to no
	edge of $M$ precolored $1$.

	As before, we color all edges of $M_1$ by its adjacent precolored edge
	of $M$. The resulting partial coloring of $G_1$ is extendable to
	a proper edge coloring of $G_1$ using colors $2,\dots, \chi'(G)+1$.
	Next, we recolor every edge of $M_1$ by the color $1$, color $G_2$
	correspondingly and finally color all edges of $M$ by a color missing
	at its endpoints so that the resulting coloring is proper and agrees with $\varphi$.
\end{proof}

For graphs with triangles we have the following variant
of Theorem \ref{th:main}.

\begin{theorem}
\label{th:maintri}
	If $G$ is a regular graph where every precoloring 
	of at most $k < \Delta(G)/2$
	independent edges
	is extendable to a $\chi'(G)$-edge coloring, then every precoloring of
	at most $k+1$ independent edges in 
	$G \square K_2$ is extendable
	to a $(\chi'(G)+1)$-edge coloring.
\end{theorem}

The proof of this theorem is virtually identical to the proof of the preceding
one. The only essential difference is that when $G$ is not triangle-free
we have to assume that at most $\Delta(G) /2$ edges are precolored
to be able to ensure that we can select independent edges in $G_1$ and $G_2$
that are adjacent to precolored edges of $M$ and also not adjacent
to any precolored edges in $G_1$ and $G_2$ (Cases 2-4 in the proof of
Theorem \ref{th:main}). We omit the details.

\section{Cartesian products with subcubic graphs}

In this section, we consider Conjecture \ref{conj:general} for subcubic graphs,
that is, graphs with maximum degree at most $3$.
First we prove that Conjecture \ref{conj:general} holds for graphs with maximum
degree two. The case of paths was considered above, so it suffices to consider cycles.
We shall need some well-known auxiliary results on list edge coloring.

\begin{lemma}
	For every path $P$, if one edge $e \in E(P)$ has a list of size 
	at least $1$ and all other 	
	edges have lists of at least $2$ colors, 
	then $P$ has a proper edge coloring using colors from the lists.
\label{Pathlemma}
\end{lemma}

\begin{lemma}
	If $L$ is a list assignment for the edges of 
	a cycle $C$, where every list has size at least two 	
	and not all edges have the same list, then $C$ has a proper edge coloring 
	using colors from the lists.
\label{claimcycle}
\end{lemma}

\begin{proposition}
\label{even_cycles}
\textit{Let $n \geq 2$ be a positive integer. If $\varphi$ is an edge precoloring of two edges in $C_{2n} \square K_2$, then $\varphi$ can be extended to a proper $3$-edge coloring of $C_{2n} \square K_2$.}
\end{proposition}
\begin{proof}
Let $M$ be matching of $C_{2n} \square K_2$, so that $C_{2n} \square K_2 -M$ consists
of two copies $C^1_{2n}$ and $C^2_{2n}$ of $C_{2n}$.
The cases when no precolored edge is contained in the matching $M$ can be dealt with as
above. If $M$ contains exactly one precolored edge, then if the two precolored edges have
distinct colors, then $\varphi$ is trivially extendable.
If, on the other hand, the two precolored have the same color, say $1$, then we can properly
color the copy of $C_{2n}$ containing the other precolored edge by colors $2$ and $3$;
so $\varphi$ is extendable.

It remains to consider the case when no precolored edges are in $C^1_{2n}$ or $C^2_{2n}$. 
If the two precolored edges have the same color, then $\varphi$ is trivially extendable.
If they have different colors and have distance at least $2$, then $\varphi$ is extendable
by Lemma \ref{claimcycle}. It is straightforward to verify that $\varphi$ is
extendable when the precolored edges are at distance $1$.
\end{proof}

Note that Proposition \ref{even_cycles} does not hold for odd cycles. 
For instance, consider the cartesian product $C_{2n+1} \square K_2$
where two corresponding edges of the two copies of $C_{2n+1}$ are colored
by $1$ and $2$, respectively. If this precoloring is extendable
to a proper $3$-edge coloring of $C_{2n+1} \square K_2$, then every edge
in the matching $M$ joining vertices of the copies of $C_{2n+1}$
must be colored $3$. Hence, the precoloring is not extendable.

Nevertheless, for odd cycles, we have the following analogue of
Proposition~\ref{even_cycles}.

\begin{proposition}
\label{odd cycles}
\textit{Let $n \geq 1$ be a positive integer. If $\varphi$ is a  edge precoloring of three 
edges in $C_{2n+1} \square K_2$, then $\varphi$ can be extended to a proper $4$-edge coloring of $G$.}
\end{proposition}

\begin{proof}
Let $M$ be a perfect matching in $C_{2n+1} \square K_2$ such that 
$C_{2n+1} \square K_2 - M$ is isomorphic to two copies $C^1_{2n+1}$ and $C^2_{2n+1}$ 
of $C_{2n+1}$. 
Without loss of generality we assume that the precoloring of $G$ uses colors $1, 2, 3$. 
We shall consider some different cases. Again, we omit the details in the
case when $M$ contains no precolored edges. \\

\noindent \textbf{Case 1.} \textit{$M$ contains exactly one precolored edge:}

For any uncolored edge $e \in E(C_{2n+1} \square K_2)$, 
we define a color list $l(e) \subseteq \{1,2,3,4\}$ by setting:

$$l(e) = \lbrace 1, 2, 3, 4 \rbrace\setminus \lbrace \varphi(e'):  e' \text{ is adjacent to } e   \rbrace.$$

If no precolored edges are in $C^2_{2n+1}$, then at most one uncolored edge 
of $C^1_{2n+1}$ is adjacent to precolored edges 
of three distinct colors, which implies that at most one edge $e$ 
has a list $l(e)$ of size $1$ and all other edges of $C^1_{2n+1}$ have lists of size at least $2$. By Lemma~\ref{Pathlemma}, there is a proper edge coloring $\varphi_1$ of $C^1_{2n+1}$,  which is an extension of the restriction of $\varphi$ to $C^1_{2n+1}$. Hence, we obtain an extension of $\varphi$ by coloring every edge of $C^2_{2n+1}$ by the color of its corresponding edge in $C^1_{2n+1}$, and then coloring every edge of $M$ with some color in $\lbrace 1, 2, 3, 4 \rbrace$ that is missing at its endpoints.

Suppose now that both $C^1_{2n+1}$ and $C^2_{2n+1}$ each contains exactly one precolored edge. Let $e_1 \in E(C^1_{2n+1})$, $e_2 \in M$, and $e_3 \in E(C^2_{2n+1})$ 
be the precolored edges of $C^1_{2n+1}$, $M$ and $C^2_{2n+1}$, respectively. 

\bigskip

First we treat the case when $n=1$;  it needs to be considered separately.

If $e_1$ and $e_2$, and $e_2$ and $e_3$ are adjacent, then $\varphi$
is extendable since $\varphi(e_2) \notin \{\varphi(e_1), \varphi(e_3)\}$;
we can e.g. first properly color $C^1_{3}$ and $C^2_{3}$ using colors from 
$\{1,2,3,4\} \setminus \{\varphi(e_2)\}$, and then color the remaining uncolored
edges of $M$.
If exactly two of the edges $e_1,e_2,e_3$ are pairwise adjacent,
say $e_1$ and $e_2$, and $\varphi(e_3) = \varphi(e_2)$, then
we can instead use colors $\varphi(e_1)$, $\varphi(e_2)$ and one additional color
from $\{1,2,3,4\}$ for coloring $C^1_{3}$ and $C^2_{3}$;
if $\varphi(e_3) \neq \varphi(e_2)$, then we proceed similarly using colors
$\{1,2,3,4\} \setminus \{\varphi(e_2)\}$.

Suppose now that all the edges $e_1, e_2, e_3$ are pairwise nonadjacent.
If at most two colors, say $c_1$ and $c_2$, appear on the precolored edges,
then we first color the uncolored edges of $C^1_{3}$ and $C^2_{3}$
using colors $\{1,2,3,4\} \setminus \{c_1,c_2\}$, and then color the remaining
uncolored edges of $M$. If, on the other hand, three colors appear
on the precolored edges, then we properly color $C^1_{3}$ and $C^2_{3}$
using colors $\{1,2,3,4\} \setminus \{\varphi(e_2)\}$. \\

Now we treat the case when $n\geq2$. We define a list 
assignment $l$ for $C^1_{2n+1}$ by setting
$$l(e) = \{ 1, 2, 3 \} \setminus 
\{\varphi(e'):  e' \in E(C_{2n+1} \square K_2) \text{ is adjacent to } e   \}.$$
 
Since at most one edge is adjacent to both $e_1$ and $e_2$, at most one uncolored edge 
$e$ satisfies that $|l(e)| =1$, and all other edges of $C^1_{2n+1}$ have lists of size at least two. By Lemma~\ref{Pathlemma}, there is a proper edge coloring $\varphi_1$ of $C^1_{2n+1}$ using colors $1, 2, 3$ which is an extension of the restriction of $\varphi$ to $C^1_{2n+1}$.  Arguing similarly, we can define a proper edge coloring $\varphi_2$ of $C^2_{2n+1}$ using colors $1, 2, 3$ which is an extension of the restriction of $\varphi$ to $C^2_{2n+1}$. Finally, we obtain an extension of $\varphi$ by coloring every uncolored edge of $M$ with color $4$. \\

\noindent \textbf{Case 2.} \textit{$M$ contains at least two precolored edges:}

For the uncolored edges of $C_{2n+1} \square K_2$, we define a list assignment $l$
by setting $$l(e) = \{1,2,3,4\} \setminus 
\{\varphi(e'):  e' \in E(C_{2n+1} \square K_2) \text{ is adjacent to } e   \}.$$
If $M$ contains exactly two precolored edges, then exactly one precolored edge 
is in $C^1_{2n+1}$ or $C^2_{2n+1}$, say $C^1_{2n+1}$. 
Then at most one edge of $C^1_{2n+1}$ is adjacent to precolored edges of three distinct colors, so at most one edge $e$ satisfies that $|l(e)| = 1$, and all other edges 
of $C^1_{2n+1}$ have lists of size at least $2$. 
Hence by Lemma~\ref{Pathlemma}, there is a proper edge coloring $\varphi_1$ of $C^1_{2n+1}$ using colors from the lists. By coloring $C^2_{2n+1}$ correspondingly, and then coloring the uncolored edges of $M$, we obtain an extension of $\varphi$.

Suppose now that $M$ contains all the three precolored edges. If all the precolored edges have the same color, then $\varphi$ is trivially extendable. If, on the other
hand, at least two colors appear on the precolored edges, then
it follows from Lemma~\ref{claimcycle}, that $\varphi$ is extendable by first
coloring the edges of $M$.
\end{proof}
Note that even though any partial $3$-edge coloring of $C_{2n+1}$ is extendable,
the upper bound of three precolored edges
in Proposition~\ref{odd cycles} is best possible. For instance, consider 
a precoloring where two adjacent edges $e_1$ and $e_2$ of $C^1_{2n+1}$ are precolored 
$1$ and $2$, respectively, and the corresponding edges $e_1'$ and $e_2'$ of $C^2_{2n+1}$ are precolored $3$ and $4$, respectively (using the same notation as in the 
preceding proof).  \\

Next, we shall verify that Conjecture \ref{conj:general} holds for Class 1 graphs
of maximum degree three.

\begin{theorem}
\label{prop:deg3}
	If $G$ is a Class 1 graph with $\Delta(G)=3$ and every partial $3$-edge 
	coloring of 
	at most $k < 3$ edges in $G$ is extendable, then every partial 
	$4$-edge coloring of at most
	$k+1$ edges in $G \square K_2$ is extendable.
\end{theorem}

\begin{proof}
	The case when $k=1$ is trivial, so let us assume that $k =2$.
	Let $G_1$ and $G_2$ be copies of $G$ in $G \square K_2$ and $M$ the
	perfect matching joining vertices of $G_1$ with corresponding vertices of $G_2$.

	Consider a partial $3$-edge coloring $\varphi$ of $G \square K_2$ with
	three precolored edges.
	Note that we may assume that $G$ is connected, since otherwise we just
	consider every component of $G$ separately. Next, we prove that
	$G$ contains no triangle.
	
	If $G$ is a triangle, then we can apply Proposition 
	\ref{odd cycles}, so we may  assume that this is not the case.
	Thus if $G$ contains a triangle $xyzx$, then at least one vertex,
	say $x$ has degree $3$. Then we can color an edge incident with $x$ not on $xyzx$
	by the color $2$, and $yz$ by the color $1$. The resulting partial edge coloring
	is not extendable to a proper $3$-edge coloring, contradicting that
	any partial edge coloring with $2$ precolored edges is extendable.
	Thus $G$ is triangle-free.
	
	In the remaining part of the  proof of Theorem \ref{prop:deg3} 
	we shall consider a large number of
	different cases. The arguments use many ideas that are similar to ones that
	have been used before. Thus, we choose to sketch the arguments, rather than giving
	all the details, when they use ideas that have been employed in the preceding proofs.
	
	\bigskip
	{\bf Case 1} {\it All precolored edges are in $G_1$ and/or $G_2$:}
	
	This case can be dealt with as in previous proofs, so we omit the details.
	
\bigskip
	{\bf Case 2.} {\it $M$ contains exactly one precolored edge:}
	
	If the color of the precolored edge $e$ of $M$ only appears on $e$ under
	$\varphi$, then the result is trivial.
	
	Suppose that $e=u_1u_2$ is a precolored edge of $M$, 
	and, without loss of generality,
	that $\varphi(e)=1$, where $u_i \in V(G_i)$, and that at least one other
	edge is precolored $1$ under $\varphi$.
	
	If the remaining two precolored edges are in $G_1$, then
	we consider the precoloring $\varphi'$ of $G_1$ obtained from the restriction
	of $\varphi$ to $G_1$ by removing color $1$ from all edges precolored $1$
	in $G_1$. Then $\varphi'$ is extendable to a proper $3$-edge coloring
	using colors $2,3,4$. Next, we recolor the edges precolored $1$
	by the color $1$, color $G_2$ correspondingly, and color every edge
	of $M$ by a color missing at its
	endpoints to obtain an extension of $\varphi$.
	
	Suppose now that both $G_1$ and $G_2$ contains precolored edges.
	If $d_{G_1}(u_1) =3$, then $d_{G_2}(u_2) =3$, and since $G-M$ contains at most
	two precolored edges and $G$ is triangle-free, there 
	are uncolored corresponding edges $e_1 \in E(G_1)$ 
	and $e_2 \in E(G_2)$,
	adjacent to $e$, such that neither $e_1$, nor $e_2$,
	is adjacent to an edge of $G-M$ precolored $1$.
	We define a new precoloring $\varphi'$ from the restricton of $\varphi$
	to $G-M$ by coloring $e_1$ and $e_2$ by the color $1$.
	The restrictions of $\varphi'$ to $G_1$ and $G_2$, respectively,
	are extendable to proper $3$-edge colorings
	of $G_1$ and $G_2$, respectively. 
	Next we color $e_1$ and $e_2$ by the color $4$,
	$e$ by the color $1$, and every uncolored edge in $M$ 
	by a color in $\{1,2,3,4\}$ missing at its endpoints.

	Now assume that $d_{G_1}(u_1) =2$. If
	there is an uncolored edge $e_1 \in E(G_1)$ adjacent to
	$e$ such that neither $e_1$, nor the corresponding edge $e_2 \in E(G_2)$,
	is adjacent to an edge of $G-M$ precolored $1$, then we proceed as in the
	preceding paragraph. 
	Otherwise, there are corresponding edges $e_1 \in E(G_1)$ and $e_2 \in E(G_2)$
	that are adjacent to $e$, and satisfying that 
	$e_1$ is precolored or adjacent to an edge precolored $1$, 
	and $e_2$ is neither precolored, 
	nor adjacent to a precolored edge of $G_2$;
	moreover, there are corresponding edges $e'_1 \in E(G_1)$ and $e'_2 \in E(G_2)$ 
	adjacent to $e$, that
	satisfy analogous
	conditions with the roles of $G_1$ and $G_2$ interchanged.
	Hence, we may color $e'_1$ and $e_2$ by colors from $\{2,3\}$ to obtain extendable
	partial proper colorings of $G_1$ and $G_2$ from the restriction of $\varphi$ to
	$G_1$ and $G_2$, respectively. Note that in extensions of these colorings no edge
	colored $1$ is adjacent to $e$. Hence, $\varphi$ is extendable.
	A similar argument applies when $d_{G_1}(u_1) = 1$.

	\bigskip
	{\bf Case 3.} {\it $M$ contains exactly two precolored edges:}
	
	We assume that $G_1$ contains the third precolored edge $e_3$.
	If all precolored edges of $M$ have the same color, then we
	proceed as in Case 2 when $E(G_1) \cup M$ contains all precolored edges.
	Thus, we assume that two different colors $1$ and $2$ appear on the precolored
	edges of $M$; let $u_1$ and $v_1$ be the endpoints of these edges in $G_1$,
	respectively, where $u_1$ is incident with an edge of $M$ precolored $1$.
	
	Suppose first that $e_3$ is colored by some color appearing on $M$, say $1$.
	If there is an uncolored edge $e'$ incident with $u_1$ that is neither incident with
	$v_1$, nor adjacent to $e_3$, then we color $e'$ by the color $1$ and take
	an extension of the coloring of $e'$ and $e_3$ using colors $1,3,4$. We now
	proceed as before to obtain an extension of $\varphi$.
	Otherwise, if there is no such edge $e'$, then
	\begin{itemize}
	
		\item[(a)] $d_{G_1}(u_1) =1$ and $u_1$ is adjacent to an endpoint of $e_3$, or
		
		\item[(b)] $d_{G_1}(u_1)=1$ and $u_1$ and $v_1$ are adjacent, or
		
		\item[(c)] $d_{G_1}(u_1) =2$ and $u_1$ is adjacent both to $v_1$
		and an endpoint of $e_3$ (and these vertices are distinct).
	
	\end{itemize}
	If (a) holds, then we may simply take an extension of the restriction of $\varphi$
	to $G_1$ using colors $1,3,4$; 
	if (b) or (c) holds, then we color $u_1v_1$ by color $3$ and take an extension 
	of the obtained coloring
	of $G_1$ using colors $1,3,4$. In all cases, it is straightforward that $\varphi$
	is extendable.
	
	Suppose now that $\varphi(e_3)=3$. If there is no uncolored edge
	incident with $u_1$ or $v_1$, then the result is trivial. Otherwise,
	if there is such an edge, then we proceed as in the preceding paragraph.

	\bigskip
	{\bf Case 4.} {\it $M$ contains exactly three precolored edges:}
	
	If all three precolored edges of $M$ have the same color, then the
	result is trivial.
	
	Suppose now that two colors appear on the precolored edges in $M$, say $1$ and $2$,
	and there is a color, say $2$, that only appears on one edge.
	Denote this edge by $e=u_1u_2$, where $u_1 \in V(G_1)$.
	If $d_{G_1}(u_1) \leq 2$, then properly color the edges incident
	with $u_1$ by $3$ and $4$. The resulting precoloring is extendable
	to a proper edge coloring of $G_1$ using colors $2,3,4$. We obtain
	an extension of $\varphi$ by coloring $G_2$ correspondingly, coloring
	$u_1u_2$ by the color $2$ and all other edges of $M$ by the color $1$.
	If $d_{G_1}(u_1) =3$, then there is an edge $u_1x$ of $G_1$ that is only
	adjacent to one precolored edge of $M$. We define a precoloring
	of $G_1$ by coloring $u_1x$ by the color $2$. This precoloring is extendable
	to a precoloring of $G_1$ by colors $2,3,4$. We now obtain an extension
	of $\varphi$ by recoloring $u_1 x$ and proceeding as before.
	
	Suppose now that three colors appear on the edges of $M$, i.e.,
	that $u_1u_2$ is colored $1$, $v_1v_2$ is colored $2$, and $w_1w_2$
	is colored $3$, where $u_i,v_i,w_i \in V(G_i)$.
	
	If there exist two distinct vertices 
	$x,y \in V(G_1) \setminus \{u_1,v_1,w_1\}$ such that
	$x$ and $y$ can be matched to distinct vertices in 
	$\{u_1,v_1, w_1\}$ by two independent
	edges, say $u_1x$ and $v_1y$, then we color these edges 
	by $1$ and $2$ respectively
	and take an extension of this precoloring of $G_1$ using colors $1,2,4$.
	Next, we recolor both $u_1x$ and $v_1y$ by the color $3$,
	color $u_1u_2$ by the color $1$, $v_1v_2$ by the color $2$ and all other
	edges of $M$ by the color $3$; this yields an extension of $\varphi$.
	Otherwise, if no such edges
	exist, then $\{u_1, v_1, w_1\}$ has at most two neighbors outside
	$\{u_1, v_1, w_1\}$ in $G_1$. Moreover, since $G_1$ is connected, there
	must be at least one neighbor of 
	$\{u_1, v_1, w_1\}$ in $V(G_1) \setminus \{u_1,v_1,w_1\}$
	in $G_1$.

	Suppose first that $\{u_1, v_1, w_1\}$ has exactly one neighbor 
	$x \notin \{u_1, v_1, w_1\}$ in $G_1$. 
	If all vertices in $\{u_1, v_1, w_1\}$ are adjacent to $x$, then
	we color $xu_1$ by $1$ and $xv_1$ by color $4$. By assumption, this 
	coloring of $G_1$ is extendable to a proper edge coloring of $G_1$
	using colors $1,2,4$. Since $G$ is triangle-free, it follows that
	$\varphi$ is extendable.
	
	If only one vertex in $\{u_1, v_1, w_1\}$ is adjacent to $x$, say $u_1$,
	then either both $v_1$ and $w_1$ are adjacent to $u_1$,
	or both $u_1$ and $w_1$ are adjacent to $v_1$. In the first case
	we color $u_1v_1$ by $4$ and $u_1w_1$ by $2$; in the latter case
	we color $v_1w_1$ by $4$ and $u_1v_1$ by $3$. Both these partial
	colorings of $G_1$ are extendable to proper edge colorings of $G_1$
	using colors $2,3,4$. Hence, $\varphi$ is extendable.
	
	Suppose now that two vertices in $\{u_1, v_1, w_1\}$
	are adjacent to $x$, say $u_1$ and $v_1$.
	If $w_1$ is adjacent to both $u_1$ and $v_1$, then we color $xv_1$ by 
	the color $1$, and $u_1w_1$ by the color $2$. This precoloring of $G_1$
	is extendable (using colors $1,2,4$), and so, $\varphi$ is extendable.
	If $w_1$ is only adjacent to one of $u_1$ and $v_1$, then we may
	proceed similarly.

	Let us finally consider the case when
	$\{u_1, v_1, w_1\}$ has exactly two neighbors $x,y \notin \{u_1, v_1, w_1\}$
	in $G_1$.
	Then $x$ and $y$ has only neighbor in $\{u_1, v_1, w_1\}$, say $u_1$. Thus
	$v_1$ and $w_1$ are only adjacent to vertices in $\{u_1, v_1, w_1\}$.
	Since $G$ is triangle-free, we can 
	properly color the edges incident with $w_1$ and $v_1$ by
	two colors from $\{2,3,4\}$
	so that no vertex is incident with two edges of the same color. 
	By assumption, there
	is an extension of the obtained coloring of $G_1$ using colors $2,3,4$. 
	Hence, $\varphi$
	is extendable. This completes the proof of the theorem.
\end{proof}

	Unfortunately, we are not able to prove a corresponding result for 
	Class 2 graphs with maximum degree $3$, since we cannot handle 
	the presence of more precolored edges using our method.
	Nevertheless, we note that since $\Delta(L(G))\leq 4$ if $G$ is $3$-regular,
	where $L(G)$ denotes the line graph of $G$,
	it follows from
	the characterization of non-degree-choosable graphs\footnote{
	A graph $G$ is {\em degree-choosable} if it has an 
	$L$-coloring whenever $L$ is a list assignment such that $|L(v)| \geq d_G(v)$
	for all $v \in V(G)$.} proved in \cite{Borodin, ERT}	that one can 
	decide in polynomial time whether
	a given partial $4$-edge coloring of a graph with maximum degree $3$ 
	is extendable to a proper $4$-edge coloring.
	In particular, any partial coloring with at most three precolored 
	edges of a subcubic Class 2 graph is extendable
	(while a precoloring of four edges is obviously not always possible to extend).
	Thus, subcubic Class 2 graphs constitutes another large family of graphs 
	which admits an Evans-type result.
	So while there are well-known examples of subcubic Class 1 graphs that
	do not admit an Evans-type result (see e.g. \cite{CasselgrenMarkstromPham}),
	such examples do not exist for Class 2 graphs.
	
	Furthermore, let us note that the condition on degree here is best 
	possible since there are $4$-regular Class 2 graphs
	where not every partial coloring of at most $4$ edges is extendable 
	to a proper $5$-edge coloring \cite{AndersenHilton2}; indeed
	$K_5$ is such a graph.

\section*{Acknowledgements}
Petros and Fufa thank the International Science Program in Uppsala, Sweden, for financial support.

Casselgren was supported by a grant from the Swedish Research council VR
(2017-05077).

\end{document}